\documentclass[12pt,reqno]{amsart}

\usepackage{xcolor}
\usepackage{amsfonts,graphicx}

\newtheorem{theorem}{Theorem}[section]
\newtheorem{lemma}[theorem]{Lemma}
\newtheorem{corollary}[theorem]{Corollary}

\def\aiOrcid{\hspace*{.4mm}\includegraphics[scale=.5]{orcid_16x16.png}}
\definecolor{orcidlogocol}{HTML}{A6CE39}
\def\orcidID#1{\href{https://orcid.org/#1}{\textcolor{orcidlogocol}{\aiOrcid} }}

\definecolor{darkblue}{rgb}{0.0, 0.0, 0}

\usepackage{tikz}
\usepackage{verbatim}
\usetikzlibrary{arrows,shapes}

\newcommand{\sq}[1]{\ensuremath \sqrt{#1}} 

\begin{document}
\title{An improved bound on the burning number of graphs}\thanks{The authors were supported by NSERC}
\author[A.\ Bonato]{Anthony Bonato}
\author[S.\ Kamali]{Shahin Kamali}
\address[A1]{Department of Mathematics, Ryerson University, Toronto, Canada.}
\address[A2]{Department of Computer Science, University of Manitoba, Winnipeg, Canada.}
\email[A1]{(A1) abonato@ryerson.ca}
\email[A2]{(A2) shahin.kamali@umanitoba.ca}

\begin{abstract}
The burning number conjecture states that the burning number of a connected graph is at most $\lceil \sqrt{n} \rceil.$ While the conjecture is unresolved, Land and Lu proved that the burning number of a connected graph is at most  $ \sqrt{(3/2)n}+O(1).$  Using an algorithmic approach, we provide an improved upper bound for the burning number of a connected graph: $$\bigg \lceil \frac{\sqrt{12n+64}+8}{3} \bigg \rceil = \sqrt{(4/3)n} +O(1) .$$
\end{abstract}

\keywords{Localization game, localization number, capture time, trees, pathwidth, projective planes, treewidth}
\subjclass{05C05, 05C99}

\maketitle

\section{Introduction}\label{intro}
	
Graph burning is a simplified model for the spread of influence in a network. Associated with the process is a parameter, the burning number, which quantifies the speed at which the influence spreads to ever vertex. The smaller the burning number is, the faster an influence can be spread in the network. Graph burning is defined as follows.  Given a finite, simple, undirected graph $G$, the burning process on $G$ is a discrete-time process. Vertices may be either unburned or burned throughout the process. Initially, in round $t=0$ all vertices are unburned. At each round $t \geq 1$, one new unburned vertex is chosen to burn, if such a vertex is available. We call such a chosen vertex a \emph{source}. If a vertex is
burned, then it remains in that state until the end of the process. Once a vertex is burned in round $t$, in round $t+1$ each of its unburned neighbors becomes burned. The process ends in a given round when all
vertices of $G$ are burned. We emphasize that sources are chosen in each round for which they are available.

The burning number corresponds to an optimal choice of sources throughout the process. The \emph{burning number} of a graph $G$, denoted by $b(G)$, is the minimum number of rounds needed for the process to end. The parameter $b(G)$ is well-defined, as in a finite graph, there are only finitely many choices for the sources. The sources that are chosen over time are referred to as a \emph{burning sequence}; a shortest such sequence is called \emph{optimal}. Hence, optimal burning sequences have length $b(G).$ Since graph burning was first introduced in \cite{BJR0,BJR,thez}, a number of results, conjectures, and algorithms have been emerged on the topic. See \cite{abburn} for a recent survey of graph burning.

The \emph{Burning Number Conjecture} (or \emph{BNC}), which first appeared in \cite{BJR}, states that for a connected graph $G$ of order $n$, $$b(G)\leq \lceil {n}^{1/2}\rceil.$$
The BNC remains one of the most difficult problems on graph burning. If the BNC holds, then it is tight as paths of order $n$ have burning number $\lceil {n}^{1/2}\rceil.$ By the Tree Reduction Theorem (see \cite{BJR}), the BNC holds if it holds for trees. Note that we require $G$ to be connected here, as otherwise, the burning number can be as large as $|V(G)|,$ as in the case for a graph with no edges.

The BNC has resisted attempts at its resolution, although various upper bounds on the burning number are known. In \cite{bessy2}, it was proved that for every connected graph $G$ of order $n$ and every $0<\epsilon<1$, $$b(G)\leq \sqrt{\frac{32n}{19(1-\epsilon)}}+\sqrt{\frac{27}{19\epsilon}}$$ and
$$b(G)\leq \sqrt{\frac{12n}{7}}+3\approx 1.309 \sqrt{n}+3.$$ These bounds were improved in \cite{LL}, who proved, up until this paper, the best known upper bound:
$$b(G)\le \bigg \lceil
\frac{\sqrt{24n+33}-3}{4} \bigg \rceil = \sqrt{(3/2)n} +O(1) \approx  1.2247\sqrt{n} + O(1).$$
While the BNC is open for general graphs, it known to hold for a number of graph classes.  A \emph{spider} is a tree with at most one vertex of degree at least $3.$ A \emph{caterpillar} is a tree where deleting all vertices of degree 1 leaves a path. As proven in \cite{bl}, spiders satisfy the conjecture. As proven in \cite{liu1} and independently in \cite{hil}, caterpillars satisfy the conjecture. In \cite{kam}, it was proven that any graph with minimum degree $\delta \ge 23$ satisfy the conjecture. Although this result encompasses a large class of graphs, it omits the class of trees.

In the present paper, we provide an improved upper bound $$b(G) \le \bigg \lceil \frac{\sqrt{12n+64}+8}{3} \bigg \rceil = \sqrt{(4/3)n} +O(1) \approx 1.1547\sqrt{n} + O(1).$$ See Theorem~\ref{th:main1} and Corollary~\ref{cor:main1}. The results in the present paper provides the best known upper bound on the burning number.

All graphs we consider are simple, finite, and undirected. The distance between vertices $u$ and $v$ is denoted by $d(u,v).$ If $G$ is a graph and $u$ is a vertex of $G$, then the \emph{eccentricity} of $u$ is defined as $\max\{d(u, v): v\in V(G)\}$. The \emph{radius} of $G$ is the minimum eccentricity over the set of all vertices in $G$. The subgraph of $G$ obtained by removing a subset $S$ from $V(G)$ is denoted $G - S$. When $S$ contains a single vertex or edge, say $x$, we write $G - x$. For background on graph theory, see \cite{west}.

\section{Improved upper bound}

We introduce an algorithmic proof to show any connected graph $G$ of order $n$ can be burned within $ \sqrt{4n/3} +O(1)$ rounds. Our algorithm fixes an arbitrary spanning tree $T$ of $G$ and burns $T$ in $\sqrt{4n/3} +O(1)$ rounds. By the time $T$ is burned, all vertices of $G$ are also burned.
	
Given a tree $T$, a burning sequence of $T$ of length $k$ is equivalent to partitioning $T$ into $k$ (possibly empty) trees $T_{k-1}, T_{k-2}, \ldots, T_0$, where $T_j$ is formed by vertices burned by the $(k-j)$th source; see~\cite{BJR}. That is, $T_j$ has radius of at most $j$ and can be burned by the fire started at the $(k-j)$th source.
	
Assume we are given a set $R = \{k-1, k-2, \ldots, 0\}$ of radii, and we want to decide whether it is possible to partition an input tree $T$ of order $n$ to subtrees with distinct radii that are bounded by elements of $R$. One elementary approach works by repeatedly extracting trees from $T$ and updating $R$ in the following manner.

\begin{enumerate}
\item Select an arbitrary vertex of $T$ as the root and let the $u$ be the vertex at the deepest level of the rooted tree.
\item Select the largest value $r \in R$, and let $T_r$ be the subtree of $T$ rooted at the $r$th ancestor of $u$ (or the entire $T$ if such ancestor does not exist).
\item Set $T = T - T_r$ and $R = R-r$ and repeat the same process until $T$ becomes empty.
\end{enumerate}

Note that $T_r$ has radius at most $r;$ that is, $T_r$ can be covered by a ball of radius $r$. Given that $T_r$ has order at least $r+1$, the total order of all trees after $k$ iterations of steps (1-3) will be at least $k+ (k-1) + \cdots+1 = k(k+1)/2$. Therefore, as long as $k \geq \frac{\sq{8n+1}+1}{2} = \sqrt{2n} +O(1)$, the extracted trees cover $T$, and we derive a burning sequence of length $k.$
In the $(k-j)$th iteration of the above algorithm, a tree of radius at most $r=j$ and of order at least $j+1$ is extracted. The value of $j$ is initially $k-1$ and is decremented in each iteration of the algorithm.

Our algorithm is built around analogous ideas as described in the previous paragraph, except that in its $(k-j)$th iteration, a tree $T^*$ of radius $r^* \leq j$ and of order at least $r^*+\lfloor j/2 \rfloor -3$ is extracted. As before, $j$ is initially $k-1$ and is decremented in each iteration of the algorithm. The extra $\lfloor j/2 \rfloor -3$ vertices that are burned in the $(k-j)$th iteration ensure a faster completion of the burning process compared to the elementary scheme described above.

Before proving the main result, we need the following lemma.
	
	\begin{lemma}\label{lem:burn}
		Given any value $\beta \in (0,1)$ and a set $R$ of distinct non-negative integers $r_1<r_2<\ldots<r_{j}$, there
		exists an $r^* \in R$ such that $$ (1-\beta)j \leq r^* \leq r_j - \beta j+3.$$
	\end{lemma}
	
	\begin{proof}
		There are at most $\lfloor j - \beta j \rfloor + 1$ members of $R$ that are smaller than $j -\beta j$, and at most $\lfloor \beta j\rfloor-2$ members of $R$ that are larger than $r_j - \beta j+3$. In total, at most $\left( \lfloor j - \beta j \rfloor + 1 \right) + \left( \lfloor \beta j\rfloor-2 \right)  \leq j-1$ members of $R$ are outside of the desired range, and at least one member of $i$ is inside the range.
	\end{proof}
	
	For example, when $\beta = 0.4$ and $R=\{0,1,\ldots,9\}$, we have $j=10$, $r_j=9$, which gives $j-\beta j = 6$, and $r_j-\beta j+3 = 8$. The desired range has three members: $r^*\in\{6,7,8\}$.

We now come to our main result.
	
	\begin{theorem}\label{th:main1}
If $T$ is a tree of order $n,$ then
$$b(T) \le \bigg \lceil \frac{\sqrt{12n+64}+8}{3} \bigg \rceil = \sqrt{4n/3} + O(1).$$
	\end{theorem}
	
	\begin{proof}
		
Throughout, we assume $T$ is rooted at an arbitrary vertex. Let $R$ be the set of radii that is initially $\{0,1, \ldots, k-1\}$. 	Let $j$ be an integer that iterates from $k-1$ to $0$.
		
		At each iteration $j$, we select a radius $r^* \in R$ and also a vertex $p\in V(T)$ so that the following two conditions hold for the tree $T_p$ rooted at $p.$

\medskip

\noindent \emph{Condition I:} The radius of $T_p$ is at most $r^*$.

\noindent \emph{Condition II:} $|V(T_p)| \geq r^*+\lfloor j/2 \rfloor -3$.

\medskip
		
		At the end of iteration $j$, we update $R$ to $R-\{r^*\}$ and $T$ to $T- T_p$ and repeat the same process in the next iteration.
		Assume we are in the beginning of the $(k-j)$th iteration of the algorithm. Suppose $R = \{r_1, r_2, \ldots, r_j\}$ at the beginning of the iteration, where $r_1 < r_2 < \ldots < r_j$.

		We will show that it is possible select $r^* \in R$ and $p\in V(T)$ so that the above two conditions hold. Let $u$ be the deepest vertex in $T$, and $v$ be the ancestor of $u$ at distance $r_j$ of $u$. Consider the subtree of $T$ rooted at $v$, denoted by $T_{v}$. 
		We consider the following cases.
		
\medskip

		\textbf{Case 1.} Suppose $ |V(T_v)| \geq r_j + \lfloor j/2 \rfloor$. In this case, we let $p = v$ and $r^* = r_j$. Note that $T_p$ (which is $T_v$) has radius at most $r_j = r^*$ (because $u$ is the deepest vertex) and Condition 1 holds. Further, by the assumption of this case, we have $|V(T_p)| = |V(T_v)| \geq r_j  + \lfloor j/2 \rfloor$, and Condition II also holds. Case 1 follows.

\medskip
		
		\textbf{Case 2.} Suppose $|V(T_v)| < r_j+\lfloor j/2 \rfloor$. Therefore, the order of $T_{v}$ is $r_j+cj$ for some $c <1/2$. By the definition of $v$, there are $r_j+1$ vertices on the path $P_{uv}$ between $u$ and $v$.
		This implies that less than $cj$ vertices of $T_v$ are not located on $P_{uv}$.

\medskip		

		We apply Lemma~\ref{lem:burn} with $\beta =1/2-c$ to select $r^* \in R$ such that
		\begin{equation}
			(1/2+c)j \leq r^* \leq r_j - (1/2-c)j+3. \label{eq:cond}
		\end{equation}
		
To choose $p$, we consider the following subcases depending on the value of $r^*.$

\medskip
		
		\textbf{Subcase 2.a.} Suppose that $r^* > r_j - j/2$.

\medskip

In this case, we again let $p = v$. The diameter of $T_p$ is at most $$r_j + cj < (r^* + j/2) + cj \leq 2r^*.$$ The last inequality follows from the lower bound for $r^*$ in (\ref{eq:cond}). Given that the diameter of $T_p$ is at most $2r^*$, its radius is at most $r^*$ and Condition I holds.
		The order of $T_p$ is $$r_j + cj \geq (r^*+(1/2-c)j-3) + cj  = r^* + j/2-3,$$ and Condition 2 also holds. The first inequality follows from the upper bound for $r^*$ in (\ref{eq:cond}). Subcase 2.a\ follows.

\medskip	
		
		\textbf{Subcase 2.b.} Suppose that $r^* \leq r_j -j/2$.

\medskip

We let $p$ to be the $h$th ancestor of $u$, where  $h = r^*+j/2$. Given that $r^*+j/2 \leq r_j$, we know that $p \in T_v$.
We have that $T_p$ has order at least $r^* + j/2$ and Condition II holds. We will show that $T_p$ has diameter at most $2r^*$. Let $(x,y)$ be the furthest pair of vertices in $T_p$ and $a$ and $b$ be the least common ancestors of $(x,u)$ and $(y,u)$, respectively. Without loss of generality, we assume the height of $a$ is no larger than that of $b$. See Figure~\ref{onre}.
\begin{figure}[!h]
  \centering
  \includegraphics{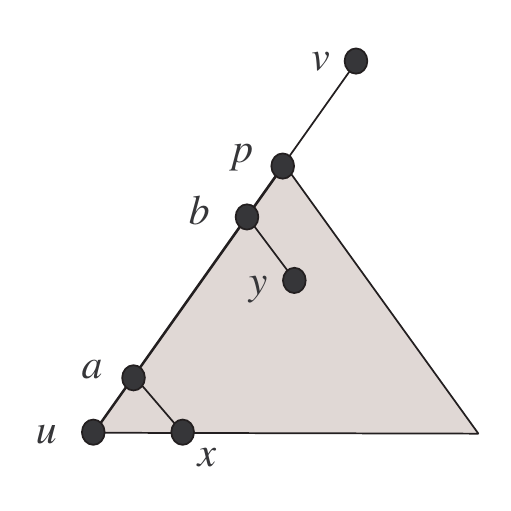}
  \caption{The subtree rooted at $p.$} \label{onre}
\end{figure}

We then have the following:
				\begin{eqnarray*}
					d(x,y) &\leq & d(x,a) + d(a,b) + d(b,y)  \\
					& \leq & d(u,a) + d(a,b) + d(b,y)  \\
					& = & d(u,b) + d(b,y)   \\
					& \leq & h + d(b,y) \\
					& \leq & h + cj  \\
					& = & r^*+j/2 + cj   \\
					& \leq & 2r^*.
				\end{eqnarray*}
The first inequality follows by the triangle inequality, the second follows since $u$ is the deepest vertex, the third inequality follows as $b \in V(T_p)$ has height $h,$ the fourth inequality follows since less than $cj$ vertices in $T_v$ are not in $P_{uv},$ and the final inequality follows by (\ref{eq:cond}). The first equality follows since $a$ and $b$ are ancestors of $u,$ and the second inequality follows by the definition of $h.$

We have established that the diameter of $T_p$ is at most $2r^*.$ Hence, the radius of $T_p$ is at most $r^*$ and Condition I also holds. Subcase 2.b follows.
		
We have therefore shown that, for a tree $T$ and any set of radii $R$, it is always possible to extract a subtree $T_p$ of $T$ with radius at most $r^* \in R$ and of order at least $r^*+\lfloor j/2 \rfloor -3$.
		Recall that $R$ is initially $\{0, 1, \ldots, k-1\}$, and $j$ iterates from $k-1$ to $0$. Therefore, after $k$ iterations, the number of vertices in all extracted trees will be at least:
		\begin{eqnarray*}
			\sum_p |V(T_p)|  & \geq  &\sum\limits_{r^*\in R} r^* + \sum\limits_{j=0}^{k-1} (\lfloor j/2 \rfloor -3) \\
			& = & \sum_{i=1}^{k-1} i + \sum_{j=1}^{k-1} \lfloor j/2 \rfloor -3k \\
            & \ge  & \sum_{i=1}^{k-1} i + \sum_{j=1}^{k-1}  j/2 - k/4  -3k \\
			&  \geq &\frac{k(k-1)}{2}  + \frac{k(k-1)}{4} -13k/4\\
            &= & \frac{3k^2-16k}{4}.
		\end{eqnarray*}
In the third line, when we remove the floors, at most $k/2$ terms are impacted and each become smaller by 1/2. Hence, in total, removing floors makes the sum larger by at most $k/4.$

As long as $n \leq \frac{3k^2-16k}{4} $, or equivalently, $k \geq \lceil \frac{\sqrt{12n+64}+8}{3} \rceil $, our algorithm partitions $T$ into subtrees with radii that are bounded by distinct values from $R = \{0,1, \ldots, k-1\}$. A tree whose radii is bounded by $r \in R$ is then burned with the $(k-r)$th source, and the burning process finishes within $k$ rounds.
	\end{proof}
	
The burning number of a connected graph $G$ equals the burning number of a spanning tree of $G;$ see \cite{BJR}. Hence, we derive the following corollary, which establishes the best known bound on the burning number of a connected graph.
	
	\begin{corollary}\label{cor:main1}
If $G$ is a connected graph of order $n, $ then
$$b(G) \le \bigg \lceil \frac{\sqrt{12n+64}+8}{3} \bigg \rceil = \sqrt{4n/3} + O(1).$$
	\end{corollary}
	
\bibliographystyle{plain}

\end{document}